\documentclass[reqno, 12pt]{amsart}
\makeatletter
\let\origsection=\section \def\section{\@ifstar{\origsection*}{\mysection}} 
\def\mysection{\@startsection{section}{1}\z@{.7\linespacing\@plus\linespacing}{.5\linespacing}{\normalfont\scshape\centering\S}}
\makeatother        
\usepackage{amsmath,amssymb,amsthm}    
\usepackage{mathrsfs}    
\usepackage{mathabx}\changenotsign

\usepackage{xcolor}
\usepackage[backref]{hyperref}
\hypersetup{
    colorlinks,
    linkcolor={red!60!black},
    citecolor={green!60!black},
    urlcolor={blue!60!black}
}
\usepackage{bookmark}
\usepackage[abbrev,msc-links,backrefs]{amsrefs}
\usepackage{doi}

\renewcommand{\PrintDOI}[1]{\doi{#1}}

\usepackage[T1]{fontenc}
\usepackage{lmodern}
\usepackage{microtype}

\usepackage[english]{babel}

\usepackage{geometry}
\usepackage{enumitem}
\def\rmlabel{\upshape({\itshape \roman*\,})}

\def\Alabel{\upshape({\itshape \Alph*\,})}

\let\polishlcross=\l
\def\l{\ifmmode\ell\else\polishlcross\fi}

\def\tand{\ \text{and}\ }

\def\qqand{\qquad\text{and}\qquad}

\let\emptyset=\varnothing
\let\setminus=\smallsetminus
\let\backslash=\smallsetminus

\makeatletter
\def\moverlay{\mathpalette\mov@rlay}
\def\mov@rlay#1#2{\leavevmode\vtop{   \baselineskip\z@skip \lineskiplimit-\maxdimen
   \ialign{\hfil$\m@th#1##$\hfil\cr#2\crcr}}}
\newcommand{\charfusion}[3][\mathord]{
    #1{\ifx#1\mathop\vphantom{#2}\fi
        \mathpalette\mov@rlay{#2\cr#3}
      }
    \ifx#1\mathop\expandafter\displaylimits\fi}
\makeatother

\newtheorem{theorem}{Theorem}[section] 
 
\newtheorem{lemma}[theorem]{Lemma}

\newtheorem{claim}{Claim}[section]

\newtheoremstyle{definition}  {4pt}  {4pt}  {\sl}  {}  {\bfseries}  {.}  {.5em}          {}
\theoremstyle{definition}
\newtheorem{definition}[theorem]{Definition}

\theoremstyle{remark}

\newtheoremstyle{introthms}  {3pt}  {3pt}  {\itshape}  {}  {\bfseries}  {.}  {.5em}          {\thmnote{#3}}\theoremstyle{introthms}

\let\eps=\varepsilon
\let\theta=\vartheta
\let\rho=\varrho
\let\phi=\varphi

\newcommand{\mpmod}[1]{(\text{mod}\,#1)}

\def\cGnk{\mathcal{G}_{n,k}}
\def\cTk{\mathcal{T}_{k}}

\begin{document}

\title[Graphs with given odd girth and large minimum degree]{On the structure of graphs with given odd girth \\ and large minimum degree}

\author[Silvia Messuti]{Silvia Messuti}
\address{Fachbereich Mathematik, Universit\"at Hamburg, Hamburg, Germany}
\email{\{silvia.messuti\,|\,schacht\}@math.uni-hamburg.de}

\author[Mathias Schacht]{Mathias Schacht}
\thanks{The second author was supported through the Heisenberg-Programme of the
DFG}

\begin{abstract}
We study minimum degree conditions for which a graph with given odd girth
has a simple structure. For example, the classical work of Andr\'asfai,
Erd\H os, and S\'os
implies that every $n$-vertex graph with odd girth $2k+1$ and minimum degree
bigger than~$\frac{2}{2k+1}n$ must be bipartite.
We consider graphs with a weaker condition on the minimum degree.
Generalizing results of H\"aggkvist and of H\"aggkvist and Jin for the cases $k=2$ and $3$,
we show that every $n$-vertex graph with odd girth $2k+1$ and
minimum degree bigger than $\frac{3}{4k}n$ is homomorphic to the cycle of length $2k+1$.
This is best possible in the sense that there are graphs with minimum degree
$\frac{3}{4k}n$ and odd girth $2k+1$ which are not homomorphic to the cycle of length $2k+1$.
Similar results were obtained by Brandt and Ribe-Baumann.
\end{abstract} 

\maketitle

\section{Introduction}
We consider finite and simple graphs without loops and for any notation not defined here we refer to
the textbooks~\cites{Bo98,BM08,Di10}.
In particular, we denote by $K_r$ the complete graph on $r$ vertices and by $C_r$ a cycle of length $r$.
A \emph{homomorphism} from a graph $G$ into a graph $H$ is a
mapping $\phi\colon V(G)\to V(H)$ with the property  that $\{\phi(u),\phi(w)\}\in E(H)$ whenever $\{u,w\}\in E(G)$.
We say that $G$ is \emph{homomorphic} to $H$ if there exists a homomorphism from~$G$ into~$H$.
Furthermore, a graph~$G$ is a \emph{blow-up} of a graph $H$, if there exists a surjective homomorphism $\phi$ from $G$ into~$H$, 
but for any proper supergraph of~$G$ on the 
same vertex set the mapping $\phi$ is not a homomorphism into $H$ anymore. In particular, a graph $G$ is homomorphic to~$H$ if and only if 
it is a subgraph of a suitable blow-up of~$H$. Moreover, we say a blow-up~$G$ of $H$ is \emph{balanced} if
the homomorphism $\phi$ signifying that $G$ is a blow-up has the additional property that 
$|\phi^{-1}(u)|=|\phi^{-1}(u')|$ for all vertices $u$ and $u'$ of $H$.

Homomorphisms can be used to capture structural properties of graphs.
For example, a graph is $k$-colourable if and only if it is homomorphic to $K_{k}$.
Furthermore many results in extremal graph theory establish relationships between 
the minimum degree of a graph and the existence of a given subgraph.
The following theorem of Andr\'asfai, Erd\H{o}s, and S\'os~\cite{AES74} is a classical result of that type.

\begin{theorem}[Andr\'asfai, Erd\H{o}s \& S\'os]
\label{thm:AES}
For every integer $r\geq3$ and for every $n$-vertex graph $G$ the following holds. If~$G$ has minimum degree $\delta(G)>\frac{3r-7}{3r-4}n$ and $G$ contains no copy of $K_r$, then $G$ is $(r-1)$-colourable.\qed
\end{theorem}

In the special case $r=3$, Theorem~\ref{thm:AES} states that every triangle-free $n$-vertex graph with 
minimum degree greater than $2n/5$ is homomorphic to $K_2$. Several extensions of this result and related questions were studied.
For example, motivated by a question of Erd\H os and Simonovits~\cite{ES73} the chromatic number of triangle-free graphs 
$G=(V,E)$ with minimum degree $\delta(G)>|V|/3$ was thoroughly investigated in~\cites{Br02,CJK97,Hae82,Ji95,Th02} and it was recently shown by Brandt and Thomass\'e~\cite{BT} that it is at most four.

Another related line of research (see, e.g.,~\cites{CJK97,Hae82,Ji95,Lu06}) 
concerned the question for which minimum degree condition a triangle-free graph~$G$ 
is homomorphic to a graph $H$ of bounded size, which is triangle-free itself. In particular,
H\"aggkvist~\cite{Hae82} showed that triangle-free graphs $G=(V,E)$ with $\delta(G)>3|V|/8$
are homomorphic to~$C_5$. In other words, such a graph $G$ is a subgraph of suitable blow-up
of $C_5$. This can be viewed as an extension of Theorem~\ref{thm:AES} for $r=3$, since balanced blow-ups of~$C_5$ show that the degree condition $\delta(G)>2|V|/5$ is sharp there.
Strengthening the assumption of triangle-freeness to graphs of higher odd girth, allows us to consider graphs with a more relaxed 
minimum degree condition. In this direction H\"aggkvist and Jin~\cite{HJ98} showed that 
graphs $G=(V,E)$ which contain no odd cycle of length three and five and with minimum degree 
$\delta(G)>|V|/4$ are homomorphic to $C_7$.

We generalize those results to arbitrary odd girth, where we say that a graph $G$ has \emph{odd girth} at least $g$, if it contains no 
odd cycle of length less than~$g$.

\begin{theorem}\label{thm:main}
	For every integer $k\geq2$ and for every $n$-vertex graph $G$ the following holds. 
	If~$G$ has minimum degree $\delta(G)>\frac{3n}{4k}$ 
	and $G$ has odd girth at least $2k+1$, then $G$ is homomorphic to $C_{2k+1}$.
\end{theorem}
Note that the degree condition given in Theorem~\ref{thm:main} is best possible as 
the following example shows. For an even integer $r\geq 6$ we denote by $M_r$ the so-called
\emph{M\"obius ladder} (see, e.g.,~\cite{GH67}), i.e., the graph obtained by adding all diagonals to a cycle of 
length~$r$, where a diagonal connects vertices of distance $r/2$ in the cycle. One may check that
$M_{4k}$ has odd girth $2k+1$, but it is not homomorphic to $C_{2k+1}$. Moreover, $M_{4k}$ is 
$3$-regular and, consequently, balanced blow-ups of $M_{4k}$ show that 
the degree condition in Theorem~\ref{thm:main} is best possible when $n$ is divisible by $4k$.

We also remark that Theorem~\ref{thm:main} implies that every graph with odd girth at least $2k+1$ 
and minimum degree bigger than $\frac{3n}{4k}$ contains an independent set of size at least $\frac{kn}{2k+1}$.
This answers affirmatively a question of Albertson, Chan, and Haas~\cite{ACH93}.
Similar results were obtained by Brandt and Ribe-Baumann (unpublished).

\section{Forbidden subgraphs}

In this section we introduce two lemmas, Lemmas~\ref{lem:C6} and~\ref{lem:3cycles} below, needed for the proof of 
Theorem~\ref{thm:main} given in Section~\ref{sec:pf_thm}.
Roughly speaking, in each lemma we show that certain configurations cannot occur in edge-maximal 
graphs considered in Theorem~\ref{thm:main}.

We say that a graph $G$ with odd girth at least $2k+1$ is \emph{edge-maximal} if adding any edge to $G$ (by keeping the same vertex set)
yields 
an odd cycle of length at most~$2k-1$. We denote by $\cGnk$ all edge-maximal $n$-vertex graphs satisfying the
assumptions of the main theorem, i.e., for integers $k\geq 2$ and $n$ we set
\[
	\cGnk\!=\!\{G\!=\!(V,E)\colon |V|\!=\!n\,,\ \delta(G)\!>\!\tfrac{3n}{4k}\,,
		\tand \text{$G$ is edge-maximal with odd girth $2k+1$}\}\,.
\]

\subsection{Cycles of length six with precisely one diagonal}
For $k$ fixed, we say an odd cycle
is \emph{short} if its length is at most $2k-1$.
A chord in a cycle of even length $2j$ is a \emph{diagonal} if it joins two vertices at distance $j$ in the cycle.
Given a walk $W$ we define its \emph{length} $\l(W)$ as the number of edges, each counted as many times as it appears in the walk.
Hence, the lengths of paths and cycles coincide with their number of edges.

\begin{lemma}\label{lem:C6}
Let $\Phi$ denote the graph obtained from $C_6$ by adding exactly one diagonal.
For all integers $k\geq2$ and $n$ and for every $G\in\mathcal{G}_{n,k}$ we have that $G$ does not contain 
an induced copy of $\Phi$.
\end{lemma}

\begin{proof}
Suppose, contrary to the assertion, that $G=(V,E)$ contains $\Phi$ in an induced way, where
$V(\Phi)=\{a_i\colon  0 \leq i \leq 5\}\subseteq V$ is the vertex set and 
\[
	E(\Phi)=\{\{a_i,a_{i+1\mpmod{5}}\}\colon 0\leq i \leq 5\}\cup\{a_1,a_4\}\,.
\]
Note that in fact, the chords of the $C_6$ in $\Phi$ which are not diagonals would 
create triangles in~$G$ so assuming that $\Phi$ is induced in $G$ gives us only information concerning the 
non-existing two diagonals.
Since $G$ is edge-maximal, the non-existence of the diagonal between~$a_0$ and $a_3$ must be forced by the existence of an even 
path~$P_{03}$ which, together with $\{a_0,a_3\}$, would yield an odd cycle of length at most $2k-1$.
Consequently, the length of $P_{03}$ is at most $2k-2$.
Since $a_0$ and $a_3$ have distance three in $\Phi$, a shortest path between them in $\Phi$, together with $P_{03}$, results in a closed walk with odd length at most $2k+1$.

Recall that any odd closed walk is either an odd cycle or it contains a shorter odd cycle, it follows that $P_{03}$ has length exactly $2k-2$ and its inner vertices are not in~$\Phi$. The same reasoning can be applied to the other missing diagonal between~$a_2$ and~$a_5$ to show that there exists another even path $P_{25}$ of length $2k-2$ whose inner vertices are disjoint from~$V(\Phi)$. 

We show that $P_{03}$ and $P_{25}$ are vertex disjoint.
Suppose that $V(P_{03})\cap V(P_{25})\neq\emptyset$ and let $b$ be the first vertex in $P_{03}$ which is also a vertex of $P_{25}$, i.e., $b$ is the only vertex from~$a_0P_{03}b$ which is also contained in $P_{25}$.
Consider the walks 
\[
	W_{05}=a_0P_{03}bP_{25}a_5\qqand
	W_{23}=a_2P_{25}bP_{03}a_3\,,
\]
where we follow the notation from~\cite{Di10}, i.e., $W_{05}$ is the walk in $G$ which starts at~$a_0$
and follows the path $P_{03}$ up to the vertex $b$ from which the walk continues on the path $P_{25}$ up to the vertex $a_5$. 
Since $W_{05}$ and $W_{23}$ consist of the same edges (with same multiplicities) as $P_{03}$ and $P_{25}$ their lengths sum up to $4k-4$. Consequently, one of the walks, say $W_{05}$, has length at most $2k-2$.
If $W_{05}$ is even, then, together with the edge $\{a_0,a_5\}$, it yields an odd closed walk of length at most $2k-1$ and hence a short odd cycle.
Otherwise, if~$W_{05}$ and $W_{23}$ are odd, then also the walks
\[
	W_{02}=a_0P_{03}bP_{25}a_2\qqand
	W_{35}=a_3P_{03}bP_{25}a_5
\]
have an odd length.
This implies that one of them, say $W_{02}$, has odd length at most~$2k-3$. Together with the path $a_0a_1a_2$ this results into a closed walk with odd length at most $2k-1$ which yields the existence of a short odd cycle. Consequently, we derive a contradiction from 
the assumption that $P_{03}$ and $P_{25}$ are not vertex-disjoint.

Having established that $V(P_{03})\cap V(P_{25})=\emptyset$, we deduce that~$G$ contains the following graph $\Phi'$ consisting
of a cycle of length $4k$ 
\[a_0a_1a_2P_{25}a_5a_4a_3P_{03}a_0\]
with three diagonals $\{a_0,a_5\}$, $\{a_1,a_4\}$, and $\{a_2,a_3\}$. 

We remark that it follows from~\cite{HJ98}*{Lemma~2}
that such a graph $\Phi'$ cannot occur as a subgraph in any $G\in\cGnk$. However, for a self contained presentation 
we include a proof below.

We show that no vertex in $G$ can be joined to four vertices in $\Phi'$.
Suppose, for a contradiction, that there exists a vertex $x$ in $G$ such that $|N_G(x)\cap V(\Phi')|\geq4$.
Recall that $x$ can be joined to at most two vertices of a cycle of length $2k+1$ and, if so,  then these vertices must have distance two in that cycle.
Since each of the three diagonals splits the cycle of length $4k$ of $\Phi'$ into two cycles of length $2k+1$, we have that $x$ cannot have more than four neighbours in $\Phi'$. Moreover, the only way to pick four neighbours is to choose two vertices from each of these cycles and none from their intersection, i.e.\ the ends of the diagonal.
By applying this argument to each of the three diagonals, we infer that no vertex from $V(\Phi)$ can be a neighbour of $x$, therefore two neighbours $b_{1}$ and $b_{2}$ are some inner vertices of $P_{03}$ and the two other neighbours $c_{1}$ and $c_{2}$ are inner vertices of $P_{25}$.
Consider the vertex disjoint paths
\[
	P_1=b_1P_{03}a_0a_1a_2P_{25}c_1 \qqand
	P_2=b_2P_{03}a_3a_4a_5P_{25}c_2\,.
\]
Since $b_1$ and $b_2$ as well as $c_1$ and $c_2$ have distance two on the cycle of length $4k$ in~$\Phi'$, both path lengths 
have the same parity and their lengths sum up to $4k-4$.
If both lengths are  odd,  one must have length at most $2k-3$ and, together with~$x$, this yields a short odd cycle.
If, on the other hand, both lengths are even, then the paths
\[
	P_1'=b_1P_{03}a_0a_5P_{25}c_2 \qqand
	P_2'=b_2P_{03}a_3a_2P_{25}c_1
\]
have odd length.
Since their lengths sum up to $4k-6$, together with $x$, this yields the existence of a short odd cycle. 
Therefore, every vertex of $G$ is joined to at most three vertices of $\Phi'$, which leads to the following contradiction
\[
	3n
	=
	4k\frac{3n}{4k}
	<
	\sum_{u\in V(\Phi')}|N_G(u)|
	=
	\sum_{x\in V}|N_G(x)\cap V(\Phi')|
	\leq 
	3|V|
	=
	3n\,.
\]
This concludes the proof of Lemma~\ref{lem:C6}.
\end{proof}

\subsection{Tetrahedra with odd faces}
In the next lemma we will show that graphs $G\in\cGnk$ contain 
no graph from the following family, which can be viewed as tetrahedra 
with three faces formed by cycles of length $2k+1$, i.e., a particular \emph{odd subdivision} 
of $K_4$ (see, e.g.,~\cite{Ge88}).

\begin{definition}[$(2k+1)$-tetrahedra]
Given $k\geq 2$ we denote by $\cTk$ the set of graphs $T$ consisting of
\begin{enumerate}[label=\rmlabel]
\item one cycle $C_T$ with three \emph{branch vertices} $a_T$, $b_T$, and $c_T\in V(C_T)$,
\item a \emph{center vertex} $z_T$, and 
\item internally vertex disjoint paths (called \emph{spokes}) 
$P_{az}$, $P_{bz}$, $P_{cz}$ connecting the branch vertices with the center. \end{enumerate}
Furthermore, we require that
each cycle in $T$ containing $z$ and exactly two of the branch vertices must have length $2k+1$ and two of the spokes have length at least two.
\end{definition}

It follows from the definition that for $T\in\cTk$ we have that 
the cycle $C_T$ has odd length and if $T\subseteq G$ for some $G\in\cGnk$, then $T$ 
consists of at least~$4k$ vertices. In fact, the length of $C_T$ equals the sum of the 
lengths of the three cycles containing~$z$ minus twice the sum of the lengths of the 
spokes. Since all three cycles containing~$z$ have an odd length, the length of $C_T$ must be 
odd as well. In particular, if $T\subseteq G$ for some $G\in\cGnk$, then the length of $C_T$ must be at 
least~$2k+1$. Summing up the lengths of all four cycles, counts every vertex twice, except the branch vertices and the center vertex, which are counted three times.
Consequently, 
\begin{equation}\label{eq:T>=4k}
	|V(T)|\geq \frac{1}{2}\big(4\cdot(2k+1)-4\big)=4k
\end{equation}
for every $T\in\cTk$ with $T\subseteq G$ for some $G\in\cGnk$.

We will also use the following further notation.
For a cycle containing distinct vertices~$u$,~$v$, and $w$ we denote by $P_{uvw}$
the unique path on the cycle with endvertices $u$ and~$w$ which contains $v$ and, 
similarly, we denote by $P_{u\overline{v}w}$ the path from $u$ to $w$ which does not 
contain $v$.

For a tetrahedron $T\in\cTk$ we denote by $C_{ab}$ the cycle containing $z$ and the two branch
vertices $a$ and $b$. Similarly, we define $C_{ac}$ and $C_{bc}$. 
Note that the union of two cycles, for instance $C_{ab}$ and $C_{ac}$, contains an even cycle 
\[
	C_{ab}\oplus C_{ac}=C_{ab}\cup C_{ac}-P_{az}
	= aP_{abz}zP_{zca}a\,,
\]
where $P_{abz}$ is a path on the cycle $C_{ab}$ and 
$P_{zca}$ a path on the cycle $C_{ac}$.
Clearly, the length of $C_{ab}\oplus C_{ac}$ equals
\begin{equation}\label{eq:oplusl}
	\l(C_{ab}\oplus C_{ac})
	=\l(C_{ab})+\l(C_{ac})-2\l(P_{az})
	=4k+2-2\l(P_{az})\,.
\end{equation}

\begin{lemma}\label{lem:3cycles}
For all integers $k\geq2$ and $n$ and for every $G\in\mathcal{G}_{n,k}$ we have that $G$ does not contain any $T\in\cTk$ as a (not necessarily induced) subgraph.
\end{lemma}

\begin{proof}
Suppose, contrary to the assertion, that $G=(V,E)$ contains a graph from~$\cTk$. 
Fix that graph $T\in\cTk$ contained in $G$ having
the shortest length of $C_T$.
We shall prove that no vertex in $G$ can be joined to four vertices in $T$ and we will obtain a contradiction to the minimum degree assumption on $G$.

Suppose that there exists a vertex $x\in V$ such that $|N_G(x)\cap V(T)|\geq4$
and fix four of those neighbours.
Since $T$ consists of the union of three cycles of length $2k+1$ one of those cycles 
must contain exactly two of these neighbours.
This implies that we can either pick two of those cycles which contain the four neighbours (see Claim~\ref{claim:4on2} below), or we have at least two ways to pick two such cycles which contain exactly three neighbours (see Claim~\ref{claim:3on2} below).

Recall that the vertices on the spokes belong to two cycles and the center~$z$ belongs to all three cycles $C_{ab}$, $C_{ac}$, and $C_{bc}$.
If $z$ is a neighbour of $x$, then one more neighbour~$z'$ must be on a spoke, because it must have distance two from $z$ and $T$ has at least two spokes of length at least two.
This means that two cycles already have two neighbours $z$ and $z'$, and the third cycle already has one neighbour, namely~$z$.
Therefore there cannot be two more neighbours of $x$ in $T$.
A similar argument shows that at most two neighbours of $x$ can lie on all the spokes of 
$T$ all together.

Before we proceed to analyze the two cases, note that $x$ can also be a vertex in~$T$. 
It is easy to check that $x$ cannot be $z$, since it would have three neighbours on the three spokes, which we just excluded.
Furthermore, $x$ cannot be one of the branch vertices.
Indeed, suppose $x=a$.
Then three neighbours~$y_1,y_2,y_3$ of $a$ are placed at distance $1$ from~$a$ on~$P_{a\overline{z}b}$, $P_{az}$ and $P_{a\overline{z}c}$ respectively, and a neighbour $y_4$ can only be on $\mathring{P}_{b\overline{z}c}$, the interior of $P_{b\overline{z}c}$.
Consider the paths
\[
	P_{24}=y_2P_{az}zP_{zby_4}y_4\qqand
	P'_{24}=y_2P_{az}zP_{zcy_4}y_4\,.
\]

Since the subpaths $zP_{zby_4}y_4$ and $zP_{zcy_4}y_4$ cover the cycle $C_{bc}$, which has length $2k+1$, the lengths of the paths $P_{24}$ and $P'_{24}$ have different parity.
Suppose that~$P_{24}$ has odd length.
Let $P_{34}$ be the path $y_3P_{acy_4}y_4$ in $C_{ac}\oplus C_{bc}$.
Then both $P_{24}$ and~$P_{34}$ have length~$2k-1$, because
\[
	\l(P_{24})+\l(P_{34})
	=\l(C_{ac}\oplus C_{bc})-2
	\overset{\eqref{eq:oplusl}}{=}4k-2\l(P_{cz})
	\leq 4k-2
\]
and together with $x$ each of the paths $P_{24}$ and $P_{34}$ create an odd cycle. 
The graph obtained from $T$ by replacing the cycle $C_{ab}$ with the cycle 
$ay_2P_{24}y_4a$ of length $2k+1$ results in a graph $T'\in\cTk$,
with branch vertices $a$, $y_4$, and $c$ and center~$z$.
Since the spoke $P_{zb}$ of $T$ 
is replaced by the larger spoke $P_{zy_4}=zP_{zby_4}y_4$ in $T'$, we have that 
the cycle $C_{T'}$ has shorter length than $C_T$.
This contradicts the choice of $T\subseteq G$.

Summarizing the above, from now on we can assume that 
$x\in V\backslash\{z,a,b,c\}$. Moreover, if $x\in V(T)$, then $x$ lies in one of the 
cycles $C_{ab}$, $C_{ac}$, or $C_{bc}$ and two of the four neighbours of 
$x$ in $T$ must be direct neighbours on this cycle.
We now consider the aforementioned cases in Claim~\ref{claim:4on2} and Claim~\ref{claim:3on2} below.

\begin{claim}\label{claim:4on2}
	Four neighbours of $x$ in $T$ are not contained in only two of the
	cycles~$C_{ab}$,~$C_{ac}$, and~$C_{bc}$.
\end{claim}

Suppose $C_{ab}$ and $C_{ac}$ contain four neighbours of $x$.
Then the spoke $P_{az}$ shared by both cycles does not contain any neighbour of~$x$.
Let $y_{1}$, $y_{2}\in N_G(x)\cap \mathring{P}_{abz}$ and $y_{3}$, $y_{4}\in N_G(x)\cap \mathring{P}_{acz}$, where 
$y_1$ and $y_3$ are the neighbours of $x$ coming first on the respective paths ($P_{abz}$ and $P_{acz}$) starting at~$a$.
Consider the paths
\[P_{13}=y_1P_{zba}aP_{acz}y_3\qqand
P_{24}=y_2P_{abz}zP_{zca}y_4\,.\]
Since the neighbours in the same $(2k+1)$-cycle have distance two and $\l(C_{ab}\oplus C_{ac})$ is even, 
we infer that $P_{13}$ and $P_{24}$ have the same parity and
\[\l(P_{13})+\l(P_{24})=2(2k+1)-2\l(P_{az})-4\leq 4k-4\,.\]
If $P_{13}$ and $P_{24}$ have odd length, then one of them must have length at most $2k-3$, thus, together with $x$, it yields the existence of a short odd cycle.
This implies that~$P_{13}$ and $P_{24}$ have even length.
Consequently, the paths
\[
	P_{14}=y_1P_{zba}aP_{az}zP_{zca}y_4\qqand
	P_{23}=y_2P_{abz}zP_{za}aP_{acz}y_3
\]
have odd length and we have that
\[\l(P_{14})+\l(P_{23})=2(2k+1)-4=4k-2\,.\]
Therefore, because of the odd girth of $G$, they must have both length $2k-1$.

Suppose that one path, say $P_{14}$, has no endpoints inside the spokes $P_{bz}$ and $P_{cz}$ 
(here the branch vertices $b$ and $c$ are allowed to be neighbours of $x$) and $x$ itself is not a vertex 
of $P_{bz}$ and $P_{cz}$.
In this case consider the $(2k+1)$-cycle $C_{y_1c}$ given by $xy_1P_{14}y_4x$.
As a result the graph obtained from $T$ by replacing $C_{ac}$ with $C_{y_1c}$ is a graph $T'\in\cTk$ with $\l(C_{T'})<\l(C_{T})$, 
since the spoke $P_{za}$ is replaced by the longer spoke $P_{zy_1}=zP_{zab}y_1$. This contradicts the choice of $T$.
Furthermore, if $x$ would be on one of the spokes $P_{bz}$ or $P_{cz}$, then it must lie on $P_{bz}$ since otherwise $x$ would lie between $y_3$ and $y_4$ and then $y_4$ would be contained in the interior of $P_{cz}$, which we excluded here.
Consequently, we arrive at the situation that $y_1=b$ and both $y_2$ and $x$ are inside $P_{bz}$. 
Hence, the four neighbours of $x$ are also contained in the cycle $C_{ac}\oplus C_{bc}$, which also contains $P_{23}$. 
Next we consider the path
\[
	P_{14}'=y_1P_{y_1ca}y_4
\]
in $C_{ac}\oplus C_{bc}$. Since $\l(C_{ac}\oplus C_{bc})$ is even and $\l(P_{23})$ is odd
we have  
\[\l(P'_{14})=\l(C_{ac}\oplus C_{bc})-\l(P_{23})-4\] 
is also odd. Recalling, that $\l(P_{23})=2k-1$
we obtain 
\[
	\l(P'_{14})
	=
	2(2k+1) - 2\l(P_{cz}) -\l(P_{23}) - 4 
	=
	2k-1-2	\l(P_{cz})
	\leq
	2k-3\,.
\]
Hence, we arrive at the contradiction that $P'_{14}$ together with~$x$ yields a short odd cycle in~$G$.
Thus both of the paths $P_{13}$ and $P_{24}$ must have an end vertex on one of the spokes~$P_{bz}$ and $P_{cz}$.
If both  paths have an end vertex on the same spoke, say $P_{bz}$, then we can repeat 
the last argument (considering $P'_{14}$).

Therefore, it must be that both $P_{bz}$ and $P_{cz}$ contain one neighbour of $x$ each, namely~$y_2$ and $y_4$.
Since $y_2$ and $y_4$ are in the same $(2k+1)$-cycle $C_{bc}$, they also have distance two in~$T$.
This means that $T$ contains a path $y_1by_2zy_4$ which, together with $x$, results in cycle $xy_1by_2zy_4x$ of length six.
Note that  the diagonal $\{y_2,x\}$ is present.
Owing to Lemma~\ref{lem:C6} at least one of the other diagonals $\{y_1,z\}$ and $\{b,y_4\}$ must be an edge of $G$.
But both these edges are chords in cycles  ($C_{ab}$ and $C_{bc}$) of length $2k+1$, 
which contradicts the odd girth assumption on $G$. This concludes the proof of Claim~\ref{claim:4on2}.

\begin{claim}\label{claim:3on2}
Three neighbours of $x$ in $T$ are not contained in only two 
	of the cycles~$C_{ab}$,~$C_{ac}$, and $C_{bc}$.
\end{claim}
Let $T\subseteq G$ chosen in the beginning of the proof violate the claim. 
First, we will show that we may assume that $T$ also 
has the following properties:
\begin{enumerate}[label=\Alabel]
\item \label{it:A} all four neighbours of $x$ are contained in $C_T$, 
\item \label{it:B} the two cycles can be chosen in such a way,
	 that the spoke shared by them contains no neighbour of $x$ and has length at least two, and
\item \label{it:C} the cycle containing one neighbour of~$x$ has the property that this 
	neighbours is not one of the two branch vertices contained in that cycle. 
\end{enumerate}

Owing to Claim~\ref{claim:4on2} we know that any pair of two out of the three cycles $C_{ab}$, $C_{ac}$, and~$C_{bc}$ contains
at most three of the four neighbours of $x$ in $T$. Consequently, the spokes~$P_{az}$,~$P_{bz}$, and~$P_{cz}$ all together 
can contain at most one neighbour of~$x$. Suppose $v$ is a neighbour of $x$ on the spoke $P_{az}$. 
Since we already showed that~$z$ cannot be a neighbour of~$x$, property~\ref{it:A} 
follows, by showing that~$v$ is not contained in $\mathring{P}_{az}$, the interior of $P_{az}$.
If $v\neq a$, then the two neighbours $y_1$ and $y_2$ of $x$ contained in $C_{ab}$ and $C_{ac}$ would have distance two
from~$v$. Consequently, $v$ would have to be a neighbour of $a$ in $P_{az}$ and $y_1$ and $y_2$ would also have to be 
neighbours of~$a$ in $T$. Hence, replacing $a$ by $x$ would give a rise to a subgraph~$T'\in \cTk$ of~$G$, where $x$
is a branch vertex. This yields a contradiction as shown before Claim~\ref{claim:4on2} and, hence, property~\ref{it:A} 
must hold.

Furthermore, if none of the neighbours is a branch vertex, then 
one cycle would contain two neighbours and the other two would contain one neighbour.
Since at least two spokes have length at least two, we can select two cycles 
containing three neighbours in such a way that properties~\ref{it:B} and~\ref{it:C} hold.

If one neighbour is a branch vertex, say $b$, then the two cycles $C_{ab}$ and $C_{bc}$ contain two neighbours 
and $C_{ac}$ contains one neighbour of~$x$. In particular the spokes $P_{az}$ and $P_{cz}$ contain no neighbour 
and one of them has length at least two. This implies that we can select one of the cycles $C_{ab}$ or $C_{bc}$
together with $C_{ac}$ such that properties~\ref{it:B} and~\ref{it:C} also hold in this case. 

Without loss of generality, we may, therefore, assume that the cycle $C_{ab}$ contains two neighbours 
$y_1$ and $y_2\in P_{a\overline{z}b}\backslash\{a\}$ (where $y_1$ is closer to $a$ and $y_2$ is closer to $b$), 
that the cycle $C_{ac}$ contains one 
neighbour $y_3\in \mathring{P}_{a\overline{z}c}$, and that the spoke $P_{az}$ has length at least two. 
In $C_{ab}\oplus C_{ac}$ we consider the paths
\[
	P_{13}=y_1P_{bac}y_3\qqand
	P_{23}=y_2P_{abz}zP_{zca}y_3\,.
\]

Since $P_{az}$ has length at least two, we have that
\[
	\l(P_{13})+\l(P_{23})=2(2k+1)-2\l(P_{az})-2\leq 4k-4\,.
\]
Therefore, if $P_{13}$ and $P_{23}$ have odd length, then one has length at most $2k-3$ and, together with $x$, it yields the existence of a short odd cycle.
This implies that $P_{13}$ and $P_{23}$ have even length.
Consequently, the paths
\[
	P_{13}'=y_1P_{baz}zP_{zca}y_3
	\qqand
	P_{23}'=y_2P_{abz}zP_{zac}y_3
\]
have odd length, and we have that 
\[
	\l(P_{13}')+\l(P_{23}')=2(2k+1)-2=4k\,.
\]
Therefore, one of these paths, say $P'_{23}$ has length $2k-1$.
Set $C_{23}=xy_2P'_{23}y_3x$. 
The graph~$T'$ obtained from $T$ by replacing $C_{ab}$ with $C_{23}$ is a again member of $\cTk$.
Since the spoke $P_{az}$ is replaced by the longer spoke $P_{y_3z}=y_3P_{caz}z$, we have $\l(C_{T'})<\l(C_T)$
This contradicts the minimal choice of $T$, which concludes the proof of Claim~\ref{claim:3on2}.

Claim~\ref{claim:3on2} yields that every vertex $x$ in $G$ is joined to at most three vertices of $T$.
Recall that every $T\in\cTk$ with $T\subseteq G$ consists of at least $4k$ vertices (see~\eqref{eq:T>=4k}).
Similarly, as in the proof of Lemma~\ref{lem:C6}, we obtain the following contradiction
\[
	3n=4k\frac{3n}{4k}<\sum_{u\in V(T)}|N_{G}(u)|=\sum_{x\in V}|N_G(v)\cap V(T)|\leq 3|V|=3n\,.
\]
\end{proof}

\section{Proof of the main result}
\label{sec:pf_thm}
In this section we deduce Theorem~\ref{thm:main} from Lemmas~\ref{lem:C6} and~\ref{lem:3cycles}.

\begin{proof}[Proof of Theorem~\ref{thm:main}]
Let $G=(V,E)$ be a graph from $\mathcal{G}_{n,k}$. We may assume that~$G$ is not a bipartite graph and we will 
show that it is a blow-up of a $(2k+1)$-cycle.

First we observe that $G$ contains a cycle of length $2k+1$.
Indeed, suppose for a contradiction that for some $\l>k$ a cycle $C=a_0\dots a_{2\l}$ is a smallest odd cycle in~$G$.
Since $G$ is edge-maximal, the non-existence of the chord $\{a_0,a_{2k}\}$ is due to the fact that it creates an odd cycle of length at most $2k-1$.
Therefore~$a_0$ and~$a_{2k}$ are linked by an even path $P$ of length at most $2k-2$ which, together with the path 
$P'=a_{2k}a_{2k+1}\dots a_{2\l}a_0$ yields the existence of an odd closed walk and, hence, of an odd cycle, of length at most~$2\l-1$, which contradicts the minimal choice of~$C$.

Let $B$ be a vertex-maximal blow-up of a $(2k+1)$-cycle contained in~$G$. Let $A_0,\dots,A_{2k}$ be its vertex classes, 
labeled in such a way that every edge of $B$ is contained in $E_G(A_i,A_{i+1})$ for some $i\in \{0,\dots,2k\}$. Here and below addition in the indices of $A$ is taken modulo~$2k+1$. Clearly, the sets $A_0,\dots,A_{2k}$ are independent sets 
in~$G$.
We will show $B=G$.
Suppose, for a contradiction, that there exists a vertex $x\in V\backslash V(B)$.
Owing to the odd girth assumption on $G$, the vertex~$x$ can have neighbours in at most two of 
the vertex classes of~$B$ and if there are two such classes, then they must be of the form 
$A_{i-1}$ and $A_{i+1}$ for some $i=0,\dots,2k$.
The following claim, which follows from Lemma~\ref{lem:C6} shows that~$x$ can have neighbours in at most 
one of the vertex classes of~$B$.

\begin{claim}\label{claim:2neighbours}
If the neighbours of $x$ in $G$ belong to exactly two vertex classes $A_{i-1}$ and $A_{i+1}$, then $x\in A_i$.
\end{claim}

Moreover, we will apply Lemma~\ref{lem:3cycles} to show that $x$ 
cannot have neighbours in only one class of~$B$. 

\begin{claim}\label{claim:1neighbour}
The neighbours of $x$ in $G$ cannot belong to exactly one vertex class $A_i$.
\end{claim}

As a consequence every $x\in V\setminus V(B)$ has no neighbour in~$B$. Therefore,
$V\setminus V(B)$ would be  disconnected from $B$, which violates the edge-maximality of~$G$. Consequently,
$V\setminus V(B)=\emptyset$ and $G=B$, which (up to the verification of Claims~\ref{claim:2neighbours} 
and~\ref{claim:1neighbour}) concludes the proof of Theorem~\ref{thm:main}.
\end{proof}

\begin{proof}[Proof of Claim~\ref{claim:2neighbours}]
Let~$x\in V$ have neighbours $a_{i-1}\in A_{i-1}$ and $a_{i+1}\in A_{i+1}$. 
In order to show that $x\in A_i$, we shall prove that $x$ is joined to all the vertices from $A_{i-1}$ and to all the vertices from $A_{i+1}$.
Suppose that this is not the case and there is some vertex $b_{i-1}\in A_{i-1}$, which is not a neighbour of $x$. The argument for the other case, when there is such a vertex in $A_{i+1}$ is identical.

Fix vertices $a_{i-2}\in A_{i-2}$ and $a_i\in A_i$ arbitrarily.
This way we fixed a cycle 
\[
	C=xa_{i+1}a_ib_{i-1}a_{i-2}a_{i-1}x
\]
of length six in~$G$.
Owing to the choice of $b_{i-1}$ the diagonal $\{x,b_{i-1}\}$ is missing in~$C$.
Moreover, the diagonal $\{a_{i+1},a_{i-2}\}$ is also not present, since together with a path from 
$a_{i-2}$ to $a_{i+1}$ through the vertex classes $A_{i-3},\dots,A_1,A_0,A_{2k-1},\dots,A_{i+2}$ it would create
an odd cycle of length~$2k-1$. On the other hand, since $B$ is a blow-up, the edge $\{a_{i},a_{i-1}\}$ is contained 
in~$B\subseteq G$, which is a diagonal in~$C$. Consequently, precisely one diagonal of $C$ is present, 
which contradicts Lemma~\ref{lem:C6}. Therefore, such a vertex $b_{i-1}$ cannot exist, which yields the claim.
\end{proof}

We will appeal to Lemma~\ref{lem:3cycles} to verify Claim~\ref{claim:1neighbour}.

\begin{proof}[Proof Claim~\ref{claim:1neighbour}]
Let $\emptyset\neq N_G(x)\cap V(B)\subseteq A_i$ and fix some neighbour $a_i$ of $x$ in~$A_i$.
Moreover, for every $j\neq i$ fix a vertex $a_{j}\in A_{j}$ arbitrarily. Since $B$ is a blow-up of~$C_{2k+1}$
those vertices span a cycle $C=a_0a_1\dots a_{2k}a_0$ of length $2k+1$. Moreover, since $x$ has no neighbours in 
$A_{i-2}\cup A_{i+2}$, the vertex $x$ is neither joined to 
$a_{i-2}$ nor to $a_{i+2}$. 

The edge-maximality of~$G\in\cGnk$ implies the existence of paths $P_{a_{i-2}x}$ and~$P_{xa_{i+2}}$ in $G$ with
an even length of at most $2k-2$.  Under all choices of such paths we pick two which minimize the number of edges
together with $C$, i.e., we pick paths~$P_{a_{i-2}x}$ and~$P_{xa_{i+2}}$ of even length at most $2k-2$ such that 
\[
	E(C)\cup E(P_{a_{i-2}x})\cup E(P_{xa_{i+2}})
\]
has minimum cardinality and we set
\[
	T=C\cup P_{a_{i-2}x}\cup P_{xa_{i+2}}\subseteq G\,.
\]
We shall show that $T$ is a tetrahedron from $\cTk$ with center vertex $a_i$. Hence, Lemma~\ref{lem:3cycles}
gives rise to a contradiction and no such vertex $x$ can exist.

Owing to the path $xa_{i}a_{i-1}a_{i-2}$ of 
length three the path $P_{a_{i-2}x}$ must have length~$2k-2$.  
Similarly, $a_{i+2}a_{i+1}a_{i}x$ yields that $P_{xa_{i+2}}$ has length $2k-2$. 
Moreover,~$P_{a_{i-2}x}$ and $P_{a_{i+2}x}$ are disjoint 
from $\{a_{i-1},a_i,a_{i+1}\}$. We set 
\[
	C'= a_{i-2}P_{a_{i-2}x}xa_{i}a_{i-1}a_{i-2}
	\qqand
	C''=a_{i+2}a_{i+1}a_{i}xP_{xa_{i+2}}a_{i+2}\,.
\]
We just showed that $C'$ and $C''$ both have length $2k+1$. In order to show that 
$T$ is a tetrahedron we have to show that the cycles $C$, $C'$, and $C'$ intersect pairwise in spokes
with center $a_i$.

Consider the intersection~$P$ of the cycles $C'$ and $C''$. We will show that $P$ is a path 
with one end vertex being $a_i$. Indeed every vertex in $a\in V(P)\setminus\{a_i\}$ is a vertex in the paths
$P_{a_{i-2}x}$ and $P_{xa_{i+2}}$. Owing to the minimal choice of $P_{a_{i-2}x}$ and~$P_{xa_{i+2}}$
it suffices to show that $a$ has the same distance to $x$ in both paths.

Suppose the distances have different parity. This implies that the closed walks 
\[
	aP_{a_{i-2}x}xP_{xa_{i+2}}a
	\qqand
	a_ia_{i-1}a_{i-2}P_{a_{i-2}x}aP_{xa_{i+2}}a_{i+2}a_{i+1}a_{i}
\]
have odd length. Since those walks cover the edges (with multiplicity) of $C'$ and $C''$ with the only exception 
of $xa_i$, the sum of their lengths is $\l(C')+\l(C'')-2$. Hence, one of the closed walks would have an odd length 
of at most $2k-1$, which yields a contradiction.
If the distances between $a$ and $x$ are different, but have the same parity, then replacing the longer path by the shorter one in the corresponding cycle yields an odd cycle of length at most $2k-1$. This again contradicts the 
assumptions on~$G$ and, hence, $P=C'\cap C''$ is indeed a path with end vertex $a_i$.

In the same way one shows that $C\cap C'$ and $C\cap C''$ are paths with end vertex~$a_i$. Since those two paths
contain $a_ia_{i-1}a_{i-2}$ and $a_{i+2}a_{i+1}a_{i}$, respectively, their length is at least two.
Therefore,~$T$ is a tetrahedron from $\cTk$ with center $a_i$ and spokes $C'\cap C''$, $C\cap C'$, and~$C\cap C''$.
\end{proof}

\section{Conluding remarks}

\subsection*{Extremal case in Theorem~\ref{thm:main}}
A more careful analysis yields that the $n$-vertex graphs with odd girth at least $2k+1$ and 
minimum degree exactly $\frac{3n}{4k}$, which are not homomorphic to $C_{2k+1}$, are 
blow-up of the M\"obius ladder~$M_{4k}$. In fact, the proofs of Lemmas~\ref{lem:C6} and~\ref{lem:3cycles} can be adjusted in such 
a way that for maximal graphs $G$ with $\delta(G)\geq \frac{3n}{4k}$ they either exclude the existence of $\Phi$ resp.\ $T$ in~$G$ or they yield a copy of $M_{4k}$ in~$G$.
In the former case, one can repeat the proof of Theorem~\ref{thm:main} based on those lemmas and obtains that~$G$ is homomorphic to $C_{2k+1}$. In the latter case, one uses the degree assumption to deduce that $G$ is isomorphic 
to a blow-up of~$M_{4k}$. The details appear in the PhD-thesis of the first author.

\subsection*{Open questions}
It would be interesting to study the situation, when we further relax the degree condition in Theorem~\ref{thm:main}.
It seems plausible that if $G$ has odd girth at least~$2k+1$ and 
$\delta(G)\geq (\frac{3}{4k}-\eps)n$ for sufficiently small $\eps>0$, then the graph~$G$ is homomorphic to~$M_{4k}$.
In fact, this seems to be true until $\delta(G)> \frac{4n}{6k-1}$. At this point blow-ups of the $(6k-1)$-cycle with all chords connecting two vertices of distance~$2k$ in the cycle added, would show that this is best possible. For $k=2$ such a result was proved by Chen, Jin, and Koh~\cite{CJK97} and for $k=3$ it was obtained 
by Brandt and Ribe-Baumann~\cite{BRB09}.

More generally, for $\l\geq 2$ and $k\geq 3$ let $F_{\l,k}$ 
be the graph obtained from a cycle of length $(2k-1)(\l-1)+2$
by adding all chords which connect vertices with  distance of the form $j(2k-1)+1$ in the cycle 
for some $j=1,\dots,\lfloor (\l-1)/2\rfloor$. 
Note that $F_{2,k}=C_{2k+1}$ and~$F_{3,k}=M_{4k}$. For every $\l\geq 2$ the graph $F_{\l,k}$ is $\l$-regular, 
has odd girth $2k+1$, and it has chromatic number three. Moreover, $F_{\l+1,k}$ is not homomorphic 
to $F_{\l,k}$, but contains it as a subgraph. 

A possible generalization of 
the known results would be the following: \textsl{if an $n$-vertex graph~$G$ has odd girth at least $2k+1$ and minimum degree
bigger than $\frac{\l n}{(2k-1)(\l-1)+2}$, then it is homomorphic to $F_{{\l-1},k}$.} However, this is known to be false 
for $k=2$ and $\l>10$, since such a graph $G$ may contain a copy of the Gr\"otzsch graph which (due to having chromatic number four) is not homomorphically embeddable into any~$F_{\l,2}$. However, in some sense this is the only exception 
for that statement. In fact, with the additional condition $\chi(G)\leq 3$ the statement is known to be true 
for~$k=2$ (see, e.g.,~\cite{CJK97}).
To our knowledge it is not known if a similar phenomenon happens for~$k>2$ and it would be interesting to study this further.

The discussion above motivates the following question, which asks for an extension of the result of \L uczak for triangle-free graphs from~\cite{Lu06}. Note that for fixed~$k$ the degree of $F_{\l,k}$ divided by its number of vertices 
tends to~$\frac{1}{2k-1}$ as $\l\to\infty$. Is it true that every $n$-vertex graph with odd girth at least $2k+1$ and 
minimum degree at least $(\frac{1}{2k-1}+\eps)n$ can be mapped homomorphically into a graph~$H$ which also 
has odd girth at least $2k+1$ and $V(H)$ is bounded by a constant $C=C(\eps)$ independent of~$n$?
\L uczak proved this for~$k=2$ and we are not aware of a counterexample for larger~$k$.


\begin{bibdiv}
\begin{biblist}

\bib{ACH93}{article}{
   author={Albertson, Michael O.},
   author={Chan, Lily},
   author={Haas, Ruth},
   title={Independence and graph homomorphisms},
   journal={J. Graph Theory},
   volume={17},
   date={1993},
   number={5},
   pages={581--588},
   issn={0364-9024},
   review={\MR{1242175}},
   doi={10.1002/jgt.3190170503},
}

\bib{AES74}{article}{
   author={Andr{\'a}sfai, B.},
   author={Erd{\H{o}}s, P.},
   author={S{\'o}s, V. T.},
   title={On the connection between chromatic number, maximal clique and
   minimal degree of a graph},
   journal={Discrete Math.},
   volume={8},
   date={1974},
   pages={205--218},
   issn={0012-365X},
   review={\MR{0340075}},
}

\bib{Bo98}{book}{
   author={Bollob{\'a}s, B{\'e}la},
   title={Modern graph theory},
   series={Graduate Texts in Mathematics},
   volume={184},
   publisher={Springer-Verlag, New York},
   date={1998},
   pages={xiv+394},
   isbn={0-387-98488-7},
   review={\MR{1633290}},
   doi={10.1007/978-1-4612-0619-4},
}

\bib{BM08}{book}{
   author={Bondy, J. A.},
   author={Murty, U. S. R.},
   title={Graph theory},
   series={Graduate Texts in Mathematics},
   volume={244},
   publisher={Springer, New York},
   date={2008},
   pages={xii+651},
   isbn={978-1-84628-969-9},
   review={\MR{2368647}},
   doi={10.1007/978-1-84628-970-5},
}

\bib{Br02}{article}{
   author={Brandt, St.},
   title={A 4-colour problem for dense triangle-free graphs},
   note={Cycles and colourings (Star\'a Lesn\'a, 1999)},
   journal={Discrete Math.},
   volume={251},
   date={2002},
   number={1-3},
   pages={33--46},
   issn={0012-365X},
   review={\MR{1904587}},
   doi={10.1016/S0012-365X(01)00340-5},
}

\bib{BRB09}{article}{
   author={Brandt, St.},
   author={Ribe-Baumann, E.},
   title={Graphs of odd girth 7 with large degree},
   conference={
      title={European Conference on Combinatorics, Graph Theory and
      Applications (EuroComb 2009)},
   },
   book={
      series={Electron. Notes Discrete Math.},
      volume={34},
      publisher={Elsevier Sci. B. V., Amsterdam},
   },
   date={2009},
   pages={89--93},
   review={\MR{2591423}},
   doi={10.1016/j.endm.2009.07.015},
}

\bib{BT}{article}{
   author={Brandt, St.},
   author={Thomass\'e, St.},
   title={Dense triangle-free graphs are four
	  colorable: A solution to the {E}rd{\H o}s-{S}imonovits problem},
      note={To appear},
}

\bib{CJK97}{article}{
   author={Chen, C. C.},
   author={Jin, G. P.},
   author={Koh, K. M.},
   title={Triangle-free graphs with large degree},
   journal={Combin. Probab. Comput.},
   volume={6},
   date={1997},
   number={4},
   pages={381--396},
   issn={0963-5483},
   review={\MR{1483425}},
   doi={10.1017/S0963548397003167},
}

\bib{Di10}{book}{
   author={Diestel, Reinhard},
   title={Graph theory},
   series={Graduate Texts in Mathematics},
   volume={173},
   edition={4},
   publisher={Springer, Heidelberg},
   date={2010},
   pages={xviii+437},
   isbn={978-3-642-14278-9},
   review={\MR{2744811}},
   doi={10.1007/978-3-642-14279-6},
}

\bib{ES73}{article}{
   author={Erd{\H{o}}s, P.},
   author={Simonovits, M.},
   title={On a valence problem in extremal graph theory},
   journal={Discrete Math.},
   volume={5},
   date={1973},
   pages={323--334},
   issn={0012-365X},
   review={\MR{0342429}},
}

\bib{Ge88}{article}{
   author={Gerards, A. M. H.},
   title={Homomorphisms of graphs into odd cycles},
   journal={J. Graph Theory},
   volume={12},
   date={1988},
   number={1},
   pages={73--83},
   issn={0364-9024},
   review={\MR{928737}},
   doi={10.1002/jgt.3190120108},
}

\bib{GH67}{article}{
   author={Guy, Richard K.},
   author={Harary, Frank},
   title={On the M\"obius ladders},
   journal={Canad. Math. Bull.},
   volume={10},
   date={1967},
   pages={493--496},
   issn={0008-4395},
   review={\MR{0224499}},
}

\bib{Hae82}{article}{
   author={H{\"a}ggkvist, Roland},
   title={Odd cycles of specified length in nonbipartite graphs},
   conference={
      title={Graph theory},
      address={Cambridge},
      date={1981},
   },
   book={
      series={North-Holland Math. Stud.},
      volume={62},
      publisher={North-Holland, Amsterdam-New York},
   },
   date={1982},
   pages={89--99},
   review={\MR{671908}},
}

\bib{HJ98}{article}{
   author={H{\"a}ggkvist, Roland},
   author={Jin, Guoping},
   title={Graphs with odd girth at least seven and high minimum degree},
   journal={Graphs Combin.},
   volume={14},
   date={1998},
   number={4},
   pages={351--362},
   issn={0911-0119},
   review={\MR{1658853}},
   doi={10.1007/s003730050023},
}

\bib{Ji95}{article}{
   author={Jin, Guo Ping},
   title={Triangle-free four-chromatic graphs},
   journal={Discrete Math.},
   volume={145},
   date={1995},
   number={1-3},
   pages={151--170},
   issn={0012-365X},
   review={\MR{1356592}},
   doi={10.1016/0012-365X(94)00063-O},
}

\bib{Lu06}{article}{
   author={{\L}uczak, Tomasz},
   title={On the structure of triangle-free graphs of large minimum degree},
   journal={Combinatorica},
   volume={26},
   date={2006},
   number={4},
   pages={489--493},
   issn={0209-9683},
   review={\MR{2260851}},
   doi={10.1007/s00493-006-0028-8},
}

\bib{Th02}{article}{
   author={Thomassen, Carsten},
   title={On the chromatic number of triangle-free graphs of large minimum
   degree},
   journal={Combinatorica},
   volume={22},
   date={2002},
   number={4},
   pages={591--596},
   issn={0209-9683},
   review={\MR{1956996}},
   doi={10.1007/s00493-002-0009-5},
}
 
\end{biblist}
\end{bibdiv}
\end{document}